\documentclass{amsart}
\usepackage[utf8]{inputenc}
\usepackage{amsmath}
\usepackage{amssymb}
\usepackage{amsthm}
\usepackage{array}
\usepackage[style=numeric,isbn=false,natbib=true,giveninits]{biblatex}
\usepackage{commath}
\usepackage{enumitem}
\usepackage{hyperref}
\usepackage{leftindex}
\usepackage{thmtools}

\calclayout

\makeatletter
\let\thmtools@orig@addtoreset\@addtoreset
\renewcommand*\@addtoreset[2]{%
  \thmtools@orig@addtoreset{#1}{#2}%
  \expandafter\xdef\csname theH#1\endcsname{%
    \expandafter\noexpand\csname theH#2\endcsname.%
    \noexpand\the\noexpand\value{#1}}%
}
\makeatother

\newcommand{\ZZ}{\mathbb{Z}}
\newcommand{\RR}{\mathbb{R}}
\newcommand{\br}{\mathrm{br}}
\newcommand{\pbr}{\mathrm{pbr}}
\DeclareMathOperator{\proj}{proj}
\DeclareMathOperator{\insn}{ins}
\DeclareMathOperator{\Ht}{ht}
\DeclareMathOperator{\Span}{span}

\newcolumntype{C}{>{$}c<{$}}

\theoremstyle{plain}
\newtheorem{theorem}{Theorem}[section]
\newtheorem{proposition}[theorem]{Proposition}
\newtheorem{lemma}[theorem]{Lemma}
\newtheorem{corollary}[theorem]{Corollary}
\newtheorem{conjecture}[theorem]{Conjecture}

\theoremstyle{definition}
\newtheorem{definition}[theorem]{Definition}
\newtheorem*{lemma*}{Lemma}
\newtheorem*{example}{Example}

\theoremstyle{remark}
\newtheorem{remark}[theorem]{Remark}

\ExplSyntaxOn
\seq_new:N \l_perm_left_seq
\seq_new:N \l_perm_right_seq
\seq_new:N \l_perm_input_seq
\cs_new:Npn \__format_signed:n #1
  {
    \int_compare:nNnTF { #1 } < { 0 }
      { \overline { \int_eval:n { 0 - ( #1 ) } } } 
      { \int_eval:n { #1 } }                   
  }
\NewDocumentCommand{\typeb}{m}
  {
    \seq_clear:N \l_perm_left_seq
    \seq_clear:N \l_perm_right_seq
    \seq_set_from_clist:Nn \l_perm_input_seq { #1 }
    \seq_map_inline:Nn \l_perm_input_seq
      {
        \seq_put_right:Nn \l_perm_right_seq { \__format_signed:n { ##1 } }
        \seq_put_left:Nn \l_perm_left_seq { \__format_signed:n { 0 - ( ##1 ) } }
      }
    \seq_use:Nn \l_perm_left_seq { ~ } \mid \seq_use:Nn \l_perm_right_seq { ~ }
  }
\ExplSyntaxOff

\title{On inversion sets of joins in weak Bruhat order}
\author{Ming Yean Lim}
\date{\today}
\address{Department of Mathematics, University of Michigan, Ann Arbor, MI 48109, USA}
\email{mylim@umich.edu}
\thanks{The author is partially supported by NSF grant DMS-2401114.}

\begin{document}

\begin{abstract}
    It is well-known that any Coxeter group $W$ is a meet-semilattice with respect to weak order. Furthermore, when $W$ is finite, it is a lattice. In the first part of this paper, we give a recursive algorithm that computes joins in weak order for arbitrary finite $W$. In the second part, we complete the proof of a conjecture of Dyer expressing the inversion set of a join in terms of the union of inversion sets of each factor, under the assumption that $W$ is finite.
\end{abstract}

\maketitle

\section{Introduction}\label{sec:intro}

It is now a well-known fact that every Coxeter group $W$ is a complete meet-semilattice with respect to the weak order \cite[Section 3.2]{bjorner2005combinatorics}. This was first proved by Bj\"orner in the 1980s. When $W$ is finite, it follows that $W$ is a lattice. The proof of this fact is abstract -- it does not lead easily to computation. In 1994, Markowsky gave a recursive algorithm \cite{MR1267688} to compute joins and meets in the weak order of the symmetric group. Very recently, Biagioli--Perrone gave an algorithm \cite{biagioli2026computing} that works for type $B$. In the first part of this paper, we generalize both algorithms to arbitrary finite Coxeter groups. In particular, in the appendix, we give an algorithm that works for type $D$ Coxeter groups, operating on the complete notation of even signed permutations. This answers a problem posed by Biagioli--Perrone \cite{biagioli2026computing}.

In a 2019 paper, Dyer \cite{MR3943754} stated many conjectures related to the weak order of Coxeter groups. Among these is a conjecture expressing the inversion set of a join $u \vee w$ in terms of the inversion sets of the factors $u, w \in W$. To state this conjecture, we need to introduce some notation. Let $\Phi$ denote the set of roots acted on by $W$, and let $\Phi^+$ denote the subset of positive roots with respect to the set of simple reflections $S \subseteq W$. Set $\Phi^- = -\Phi^+$. For $w \in W$, define
\begin{equation*}
    \Phi_w = \Phi^+ \cap w(\Phi^-).
\end{equation*}
Thus, the \emph{(right) weak order} on $W$ is defined by declaring that
\begin{equation*}
    u \le w \quad \text{ if and only if } \quad \Phi_u \subseteq \Phi_w.
\end{equation*}

\begin{conjecture}[Dyer]\label{conj:dyer_finite}
    Let $W$ be a finite Coxeter group. Then for any $u, w \in W$, we have $\Phi_{u \vee w} = \overline{\Phi_u \cup \Phi_w}^\br$.
\end{conjecture}
Here $\overline{A}^\br$ denotes the \emph{Bruhat preclosure} of $A \subseteq \Phi^+$, studied by Dermenjian in \cite{dermenjian2025bruhatpreclosure}. By definition, it consists of those $\alpha \in \Phi^+$ for which there exists $\alpha_1, \ldots, \alpha_m \in A$ such that $s_\alpha = s_{\alpha_1} \ldots s_{\alpha_m}$, and
\begin{equation*}
    0 < \ell(s_{\alpha_1}) < \ell(s_{\alpha_1} s_{\alpha_2}) < \ldots < \ell(s_{\alpha_1} s_{\alpha_2} \ldots s_{\alpha_m}).
\end{equation*}
This is \cite[Conjecture 2.8]{MR3943754} for \emph{finite} Coxeter groups.

There has been some recent progress on \autoref{conj:dyer_finite}. Biagioli--Perrone proved it for types $A$ and $I$ in \cite{biagioli2026conjecture}, and they announced a proof for type $B$ in \cite{biagioli2026computing}. Moreover, they also claim \cite{biagioli2026computing} to have verified it computationally for types $F_4, H_3, H_4$. On the other hand, Dermenjian \cite{dermenjian2025bruhatpreclosure} proved that the reverse inclusion in \autoref{conj:dyer_finite} always holds. Using results in the second part of this paper, we prove
\begin{theorem}\label{thm:dyer_simply_laced}
    Let $W$ be a finite Coxeter group of simply-laced type. Then for any $u, w \in W$, we have $\Phi_{u \vee w} = \overline{\Phi_u \cup \Phi_w}^\br$.
\end{theorem}
In particular, the above result applies to types $A$, $D$, and $E$. For the latter two cases this was not previously known. Consequently, by the classification of finite irreducible Coxeter groups,
\begin{corollary}
    \autoref{conj:dyer_finite} is true.
\end{corollary}

We remark that the original formulation of Dyer's conjecture does not impose any finiteness condition on $W$. It is formulated in terms of biclosed sets, which in the finite case are precisely the inversion sets $\Phi_w$, see \cite{MR3943754} for details. This conjecture remains unresolved in general.

En route to proving \autoref{thm:dyer_simply_laced}, we define the \emph{palindromic Bruhat preclosure} $\overline{A}^\pbr$ of $A \subseteq \Phi^+$ to be the subset of $\overline{A}^\br$ where we further impose the condition that the sequence $\alpha_1, \ldots, \alpha_m$ has to be palindromic. It turns out that palindromic Bruhat preclosure is always a subset of Dyer's \emph{$2$-closure} $\overline{A}$, as defined in \cite[Section 1.4]{MR3943754}:

\begin{theorem}\label{thm:pbr_subseteq_2_closure_root}
    For any $A \subseteq \Phi^+$, we have $A \subseteq \overline{A}^\pbr \subseteq \overline{A}$.
\end{theorem}

Perhaps surprisingly, for simply-laced types, the palindromic Bruhat preclosure is precisely equal to the $2$-closure. This gives a generalization of \cite[Theorem 4.4]{biagioli2026conjecture}:
\begin{theorem}\label{thm:simply_laced_pbr_eq_2_closure}
    Let $W$ be a finite Coxeter group of simply-laced type. Then for any $A \subseteq \Phi^+$, we have $\overline{A}^\pbr = \overline{A}$.
\end{theorem}

\autoref{thm:dyer_simply_laced} then follows easily from the above theorem and previous results of Dyer and Dermenjian:
\begin{proof}[Proof of \autoref{thm:dyer_simply_laced}]
    By \cite[Theorem 1.5]{MR3943754} and \autoref{thm:simply_laced_pbr_eq_2_closure},
    \begin{equation*}
        \Phi_{u \vee w} = \overline{\Phi_u \cup \Phi_w} = \overline{\Phi_u \cup \Phi_w}^\pbr \subseteq \overline{\Phi_u \cup \Phi_w}^\br.
    \end{equation*}
    The reverse inclusion follows from \cite[Lemma 6.9]{dermenjian2025bruhatpreclosure}.
\end{proof}

\section{Computing joins in weak order}

In this section, we describe an algorithm to compute joins in the weak order of finite Coxeter groups $W$. This algorithm is recursive on parabolic subgroups of $W$, based on the following result proven in \cite[Proposition 6.3]{MR2209128} and \cite[Section 4]{MR2149070}:

\begin{theorem}\label{thm:proj_is_lattice_hom}
    Let $(W, S)$ be a finite Coxeter system and let $I \subseteq S$. Then the projection map $\proj^S_I : W \to W_I$ given by $w \mapsto w_I$ preserves meets and joins.
\end{theorem}

For the basics of Coxeter groups, we refer the reader to \cite{bjorner2005combinatorics, humphreys1990reflection}. In subsections \ref{subsec:parabolic} and \ref{subsec:poset}, we recall some notions related to parabolic subgroups and posets to clarify the statement of the above theorem, which we give a short proof of in \autoref{subsec:ins_proj}.

The relevance of the theorem is the following: Let $L, L'$ be join semi-lattices and suppose that we have a join-preserving function $f : L \to L'$. Suppose that we know how to compute both joins in $L'$ and fibers of the map $f$. Then for $x, y \in L$, we have $x \vee y \in f^{-1}(\{f(x) \vee f(y)\})$. In fact, $x \vee y$ is the minimal element $z \in f^{-1}(\{f(x) \vee f(y)\})$ such that $x \le z$ and $y \le z$.

In \autoref{subsec:fibers}, we analyze the fibers of $\proj^S_I$, then give the general algorithm in \autoref{subsec:algorithm}.

\subsection{Parabolic subgroups}\label{subsec:parabolic}

Let $\Delta \subseteq \Phi^+$ denote the set of simple roots corresponding to $S$. For a subset $I \subseteq S$, let $\Delta_I \subseteq \Delta$ denote the corresponding subset of simple roots. Let $\Phi_I = \Phi \cap \Span \Delta_I$, which is a root system with base $\Delta_I$, and let $\Phi_I^+$ denote its positive roots, and $\Phi_I^- = -\Phi_I^+$. Let $W_I$ denote the standard parabolic subgroup of $W$ generated by $I$. If $W_I$ is finite, it has a unique longest element $w_0(I)$. We also set $w_0 = w_0(S)$ if it exists.

Let $\leftindex^I W = \{w \in W \mid \forall s \in I, \ell(s w) > \ell(w) \}$. For $w \in W$, there are unique $w_I \in W_I$ and $\leftindex^I w \in \leftindex^I W$, such that $w = w_I \leftindex^I w$. In this case, $\ell(w) = \ell(w_I) + \ell(\leftindex^I w)$.

\begin{remark}
    The parabolic projection we use is the ``mirrored'' version as in \cite[Equation 2.12]{bjorner2005combinatorics}. As we shall see, this is the variant that behaves well with respect to right weak order.
\end{remark}

\begin{remark}
    For $w \in W_I$, we could have also defined $(\Phi_I)_w = \Phi_I^+ \cap w(\Phi_I^-)$. This is unnecessary, as we have $(\Phi_I)_w = \Phi_w$. Indeed, it is clear that $(\Phi_I)_w \subseteq \Phi_w$, and both sets have the same cardinality $\ell(w)$.
\end{remark}

\subsection{Posets}\label{subsec:poset}

Let $P, Q$ be posets. Recall that a function $f : P \to Q$ is \emph{monotonic} if $u \le w$ implies $f(u) \le f(w)$. We say that $f$ \emph{preserves meets} if $f(u \wedge w) = f(u) \wedge f(w)$ for all $u, w \in P$ whenever the meet $u \wedge w$ exists. Similarly, $f$ \emph{preserves joins} if $f(u \vee w) = f(u) \vee f(w)$ for all $u, w \in P$ whenever the join $u \vee w$ exists.

\begin{definition}
    Let $P, Q$ be posets, and let $L : P \to Q$ and $R : Q \to P$ be two monotonic functions. We call the pair $(L, R)$ a \emph{Galois connection} if for all $p \in P$ and $q \in Q$,
    \begin{equation*}
        L(p) \le q \quad\text{ if and only if }\quad p \le R(q).
    \end{equation*}
    In this case, $L$ is \emph{left adjoint} to $R$, and $R$ is \emph{right adjoint} to $L$.
\end{definition}

The following lemma gives a useful criterion to show that a given function preserves meets or joins:
\begin{lemma}\label{lem:adjoints_preserve_limits}
    Let $P, Q$ be posets, and let $(L, R)$ be a Galois connection as above. Then $L$ preserves joins, and $R$ preserves meets.
\end{lemma}
\begin{proof}
    This is a special case of the preservation of limits by adjoints \cite[Section IV.5]{mac1971categories}.
\end{proof}

\subsection{Parabolic insertion and projection}\label{subsec:ins_proj}

We let $(W, S)$ denote a finite Coxeter system. The finiteness hypothesis is not necessary for some of the statements that follow, but we include it for simplicity. Let $I \subseteq S$. We let $\iota^S_I : W_I \to W$ denote the inclusion map. Define the \emph{projection map} $\proj^S_I : W \to W_I$ by
\begin{equation*}
   \proj^S_I(w) = w_I.
\end{equation*}
We also define the \emph{insertion map} $\insn^S_I : W_I \to W$ by
\begin{equation*}
   \insn^S_I(w) = w w_0(I) w_0.
\end{equation*}

\begin{lemma}\label{lem:Phi_insertion}
    For $w \in W_I$, we have $\Phi_{\insn^S_I(w)} = \Phi_w \cup (\Phi^+ \setminus \Phi_I^+)$.
\end{lemma}
\begin{proof}
    We compute that $\Phi_{w w_0(I) w_0} = \Phi^+ \setminus \Phi_{w w_0(I)} = \Phi^+ \setminus (\Phi_I^+ \setminus \Phi_w) = \Phi_w \cup (\Phi^+ \setminus \Phi_I^+)$.
\end{proof}

\begin{lemma}\label{lem:Phi_projection}
    For $w \in W$, we have $\Phi_{\proj^S_I(w)} = \Phi_w \cap \Phi_I^+$.
\end{lemma}
\begin{proof}
    By \cite[Lemma 4.1 (vi)]{MR3943754}, we have
    \begin{equation*}
        \Phi_w = \Phi_{w_I \leftindex^I w} = \Phi_{w_I} \sqcup w_I (\Phi_{\leftindex^I w}).
    \end{equation*}
    It suffices to show that $w_I (\Phi_{\leftindex^I w}) \cap \Phi_I = \emptyset$. We have equivalences
    \begin{align*}
        w_I (\Phi_{\leftindex^I w}) \cap \Phi_I = \emptyset
        &\iff \Phi_{\leftindex^I w} \cap \Phi_I = \emptyset
        \iff \Phi_I^+ \cap \leftindex^I w(\Phi^-) = \emptyset \\
        &\iff (\leftindex^I w)^{-1}(\Phi_I^+) \subseteq \Phi^+
        \iff (\leftindex^I w)^{-1}(\Delta_I) \subseteq \Phi^+,
    \end{align*}
    which holds since $\leftindex^I w \in \leftindex^I W$.
\end{proof}

The following results follow easily from the previous lemmas.
\begin{proposition}
    The maps $\iota^S_I$, $\proj^S_I$, and $\insn^S_I$ are monotonic.
\end{proposition}
\begin{proof}
    If $u, w \in W$ are such that $\Phi_u \subseteq \Phi_w$, then $\Phi_u \cap \Phi_I^+ \subseteq \Phi_w \cap \Phi_I^+$ and $\Phi_u \cup (\Phi^+ \setminus \Phi_I^+) \subseteq \Phi_w \cup (\Phi^+ \setminus \Phi_I^+)$.
\end{proof}

\begin{proposition}
    The pairs $(\iota^S_I, \proj^S_I)$ and $(\proj^S_I, \insn^S_I)$ are Galois connections.
\end{proposition}
\begin{proof}
    Let $u \in W_I$ and $w \in W$. Then $\Phi_u \subseteq \Phi_w$ if and only if $\Phi_u \subseteq \Phi_w \cap \Phi_I^+$. Similarly, $\Phi_w \cap \Phi_I^+ \subseteq \Phi_u$ if and only if $\Phi_w \subseteq \Phi_u \cup (\Phi^+ \setminus \Phi_I^+)$.
\end{proof}

The above proposition together with \autoref{lem:adjoints_preserve_limits} implies \autoref{thm:proj_is_lattice_hom}: the map $\proj^S_I$ preserves meets and joins.


\subsection{Fibers}\label{subsec:fibers}

We now analyze the fibers of the projection map. We continue to assume that $W$ is finite.

\begin{lemma}
    We have that $\leftindex^I W = [e, w_0(I) w_0]$ is a closed interval in weak order.
\end{lemma}
\begin{proof}
    For $w \in W$, we have
    \begin{equation*}
        w \in \leftindex^I W
        \iff \proj^S_I(w) = e
        \iff \proj^S_I(w) \le e
        \iff w \le \insn^S_I(e). \qedhere
    \end{equation*}
\end{proof}

\begin{remark}\label{rem:chain_iff}
    We make the following observations:
    \begin{enumerate}[label=(\roman*)]
        \item For $w \in W_I$, $(\proj^S_I)^{-1}(\{w\}) = w \leftindex^I W = [w, w w_0(I) w_0]$ is a graded poset of rank $\abs{\Phi^+ \setminus \Phi_I^+}$ and cardinality $[W : W_I]$.
        \item The interval $[e, w_0(I) w_0]$ is a chain if and only if $\abs{\Phi^+ \setminus \Phi_I^+} + 1 = [W : W_I]$. This coincidence explains why we get chains in types $A$ and $B$, but not in type $D$.
    \end{enumerate}
\end{remark}

\begin{lemma}\label{lem:le_join_iff}
    Let $u, w \in W$ and set $j' = \proj^S_I(u) \vee \proj^S_I(w)$. For $j \in (\proj^S_I)^{-1}(\{j'\})$, we have $u \vee w \le j$ if and only if $(\Phi_u \cup \Phi_w) \setminus \Phi_I^+ \subseteq \Phi_j \setminus \Phi_I^+$.
\end{lemma}
\begin{proof}
    It suffices to show that $u \le j$ if and only if $\Phi_u \setminus \Phi_I^+ \subseteq \Phi_j \setminus \Phi_I^+$. We have
    \begin{equation*}
        \proj^S_I(u) \le \proj^S_I(u) \vee \proj^S_I(w) = j' = \proj^S_I(j),
    \end{equation*}
    so $\Phi_u \cap \Phi_I^+ \subseteq \Phi_j \cap \Phi_I^+$. The result is now clear.
\end{proof}

\subsection{The algorithm}\label{subsec:algorithm}

We now describe an algorithm for computing joins in finite Coxeter groups $W$, based on the approach outlined at the start of this section.

As precomputation, we first choose an ordering $S = \{s_1, \ldots, s_n\}$ and for $0 \le k \le n$ define $I_k = \{s_1, \ldots, s_k\}$. Now compute the set of elements $\leftindex^{I_k} (W_{I_{k+1}}) = [e, w_0(I_{k}) w_0(I_{k+1})]$ for $0 \le k \le n-1$. This can be done for example by repeatedly `descending' from $w_0(I_{k}) w_0(I_{k+1})$. Choose a linear ordering $\prec$ on $\leftindex^{I_k} (W_{I_{k+1}})$ extending the weak order.

Given $u, w \in W$, we define $j = u \vee w$, and for $0 \le k \le n$, define $u_k = \proj^S_{I_k}(u), w_k = \proj^S_{I_k}(w), j_k = \proj^S_{I_k}(j)$. We can compute $u_k$ and $w_k$ recursively using $u_n = u, w_n = w$ and $u_k = \proj^{I_{k+1}}_{I_k}(u_{k+1})$ and $w_k = \proj^{I_{k+1}}_{I_k}(w_{k+1})$ for $k = n-1, n-2, \ldots, 0$. Now $j_0 = e$, and for $1 \le k \le n$, $\proj^{I_k}_{I_{k-1}}(j_k) = j_{k-1}$, so $j_k = j_{k-1} z$ for some $z \in \leftindex^{I_{k-1}} (W_{I_k})$. In fact, $z \in \leftindex^{I_{k-1}} (W_{I_k})$ is the minimal element with respect to $\prec$ such that $u_k \le j_{k-1} z$ and $w_k \le j_{k-1} z$. This condition can be checked using \autoref{lem:le_join_iff}. Finally, we get $u \vee w = j = j_n$.

\begin{remark}
    The choice of ordering of $S$ does not affect the correctness of the algorithm but does affect its run-time. In practice, we choose the ordering so that
\begin{equation*}
    \sum_{k=0}^{n-1} [W_{I_{k+1}} : W_{I_k}]
\end{equation*}
is small.
\end{remark}

\section{Palindromic Bruhat preclosure}\label{sec:pbr}

For subsets $A \subseteq \Phi^+$, we have mentioned various notions of closure in the introduction, namely the $2$-closure $\overline{A}$, Bruhat preclosure $\overline{A}^\br$, and palindromic Bruhat preclosure $\overline{A}^\pbr$. We recall from \cite{MR3943754} that a subset $A \subseteq \Phi^+$ is \emph{$2$-closed} if for all $\alpha, \beta \in A$ and $a, b \in \RR_{\ge 0}$ such that $a \alpha + b \beta \in \Phi^+$, we have $a \alpha + b \beta \in A$. The \emph{$2$-closure} of $A \subseteq \Phi^+$ is the intersection of all $2$-closed subsets of $\Phi^+$ containing $A$.

We can transport these notions of closure to the set of reflections $T = \bigcup_{w \in W} w S w^{-1}$, via the bijection $\Phi^+ \to T$ given by $\alpha \mapsto s_\alpha$. For convenience, we will restate some of the definitions in terms of reflections here:

\begin{definition}[\protect{\cite{dermenjian2025bruhatpreclosure}}]
    Let $A \subseteq T$. The \emph{Bruhat preclosure} $\overline{A}^\br$ of $A$ is the set of $t \in T$ such that there exists a sequence $t_1, t_2, \ldots, t_m \in A$ with
    \begin{equation}\label{eqn:t_eq_prod}
        t = t_1 t_2 \ldots t_m,
    \end{equation}
    and for $0 \le r \le m-1$, the prefix length-increasing condition
    \begin{equation}\label{eqn:incr_len}
        \ell(t_1 t_2 \ldots t_r) < \ell(t_1 t_2 \ldots t_r t_{r+1})
    \end{equation}
    holds.
\end{definition}

\begin{definition}
    Let $A \subseteq T$. The \emph{palindromic Bruhat preclosure} $\overline{A}^\pbr$ of $A$ is the set of $t \in T$ such that there exists a sequence $t_1, t_2, \ldots, t_m \in A$ satisfying \eqref{eqn:t_eq_prod} and \eqref{eqn:incr_len} for $0 \le r \le m-1$, and furthermore $t_r = t_{m+1-r}$ for $1 \le r \le m$.
\end{definition}

Note that the $m$ in the above definitions are necessarily odd. Bruhat preclosure is a \emph{preclosure operator}, meaning that for any $A \subseteq T$, we have $A \subseteq \overline{A}^\br$, and for any $A \subseteq B \subseteq T$, we have $\overline{A}^\br \subseteq \overline{B}^\br$. Similarly, palindromic Bruhat preclosure is also a preclosure operator.

\subsection{Proof of \autoref{thm:pbr_subseteq_2_closure_root}}

In this section, we prove \autoref{thm:pbr_subseteq_2_closure_root}, or equivalently the corresponding statement for subsets of reflections instead of positive roots:
\begin{theorem}\label{thm:pbr_subseteq_2_closure}
    For any $A \subseteq T$, we have $A \subseteq \overline{A}^\pbr \subseteq \overline{A}$.
\end{theorem}

We record the following immediate corollaries:
\begin{corollary}\label{cor:pbr_of_2_closed}
    If $A \subseteq T$ is $2$-closed, then $\overline{A}^\pbr = A$.
\end{corollary}
Compare this with \cite[Theorem 4.21]{dermenjian2025bruhatpreclosure}.

\begin{corollary}\label{cor:2_closure_of_pbr}
    For any $A \subseteq T$, we have $\overline{\overline{A}^\pbr} = \overline{A}$.
\end{corollary}

We now prove \autoref{thm:pbr_subseteq_2_closure}. The following hypotheses will be in force until the end of this subsection: Let $t \in T$ and let $t_1, \ldots, t_l \in T$ be such that $t = t_1 t_2 \ldots t_{2l-1}$, where for $l+1 \le r \le 2l-1$, we set $t_r = t_{2l-r}$. Let $\alpha \in \Phi^+$ be such that $t = s_\alpha$, and for $1 \le r \le l$, let $\alpha_r \in \Phi^+$ be such that $t_r = s_{\alpha_r}$.

The following lemma gives a translation between the prefix length-increasing condition and a membership condition on roots.
\begin{lemma}\label{lem:incr_len_iff}
    The following are equivalent:
    \begin{enumerate}[label=(\roman*)]
        \item for $0 \le r \le 2l-2$, the inequality \eqref{eqn:incr_len} holds;
        \item for $0 \le r \le l-1$, $t_1 \ldots t_r(\alpha_{r+1}) \in \Phi_t$.
    \end{enumerate}
\end{lemma}
\begin{proof}
    For $0 \le r \le l-1$, condition \eqref{eqn:incr_len} is equivalent to
    \begin{equation}\label{eqn:prod_pos}
        t_1 \ldots t_r(\alpha_{r+1}) \in \Phi^+.
    \end{equation}
    For $l-1 \le r \le 2l-2$, condition \eqref{eqn:incr_len} is equivalent to
    \begin{equation*}
        \ell(t t_1 t_2 \ldots t_{r'+1}) = \ell(t_1 t_2 \ldots t_r) < \ell(t_1 t_2 \ldots t_r t_{r+1}) = \ell(t t_1 t_2 \ldots t_{r'}),
    \end{equation*}
    where we set $r' = 2l-2-r$, so that $0 \le r' \le l - 1$. This is in turn equivalent to
    \begin{equation*}
        t t_1 t_2 \ldots t_{r'}(\alpha_{r'+1}) \in \Phi^-.
    \end{equation*}
    Making a change of variables, the validity of \eqref{eqn:incr_len} in the above range is equivalent to
    \begin{equation}\label{eqn:prod_neg}
        t t_1 \ldots t_r(\alpha_{r+1}) \in \Phi^-
    \end{equation}
    for $0 \le r \le l-1$.
    We conclude by noting that \eqref{eqn:prod_pos} and \eqref{eqn:prod_neg} is equivalent to (ii).
\end{proof}

Let $V$ denote the ambient real vector space containing the root system $\Phi$, with symmetric bilinear form $(\cdot, \cdot)$. Let $\langle \cdot, \cdot \rangle : V \times V^* \to \RR$ denote the canonical pairing. For $\alpha \in \Phi$, define its \emph{coroot} to be $\alpha^\vee \in V^*$ such that for $v \in V$, $\langle v, \alpha^\vee \rangle = \frac{2 (v, \alpha)}{(\alpha, \alpha)}$. Thus $s_\alpha \in T$ acts on $v \in V$ via
\begin{equation*}
    s_\alpha(v) = v - \langle v, \alpha^\vee \rangle \alpha,
\end{equation*}
and the bilinear form $(\cdot, \cdot)$ is $W$-invariant.

\begin{lemma}\label{lem:suff_prod}
    If the inequalities \eqref{eqn:incr_len} hold for $0 \le r \le 2l-2$, then
    \begin{enumerate}[label=(\roman*)]
        \item $t_{l-1} \ldots t_1(\alpha) = \alpha_l$;
        \item for $0 \le r \le l-1$, we have
            \begin{equation*}
                t_r \ldots t_1(\alpha) = t_{r+1} \ldots t_1(\alpha) + \langle t_r \ldots t_1(\alpha), \alpha_{r+1}^\vee \rangle \alpha_{r+1},
            \end{equation*}
            with $\langle t_r \ldots t_1(\alpha), \alpha_{r+1}^\vee \rangle > 0$;
        \item for $0 \le r \le l-1$, $t_r \ldots t_1(\alpha) \in \Phi^+$.
    \end{enumerate}
\end{lemma}
\begin{proof}
    To show (i), we have $s_\alpha = t = s_{t_1 \ldots t_{l-1}(\alpha_l)}$, so $\alpha = \pm t_1 \ldots t_{l-1}(\alpha_l)$. Equation \eqref{eqn:prod_pos} implies that the sign is positive.

    Now we show (ii). Let $0 \le r \le l-1$. We have
    \begin{equation*}
        t t_1 \ldots t_r(\alpha_{r+1}) = s_\alpha(t_1 \ldots t_r(\alpha_{r+1})) = t_1 \ldots t_r(\alpha_{r+1}) - \langle t_1 \ldots t_r(\alpha_{r+1}), \alpha^\vee \rangle \alpha.
    \end{equation*}
    By \eqref{eqn:prod_pos} and \eqref{eqn:prod_neg}, we must have $\langle t_1 \ldots t_r(\alpha_{r+1}), \alpha^\vee \rangle > 0$. This implies
    \begin{equation*}
        \langle t_r \ldots t_1(\alpha), \alpha_{r+1}^\vee \rangle = \langle \alpha, t_1 \ldots t_r(\alpha_{r+1})^\vee \rangle > 0.
    \end{equation*}
    Rearranging
    \begin{equation*}
        t_{r+1} \ldots t_1(\alpha) = s_{\alpha_{r+1}} (t_r \ldots t_1(\alpha)) = t_r \ldots t_1(\alpha) - \langle t_r \ldots t_1(\alpha), \alpha_{r+1}^\vee \rangle \alpha_{r+1}
    \end{equation*}
    gives the result.

    The $r=l-1$ case of (iii) follows from (i). For $r < l-1$, we induct using (ii).
\end{proof}

\begin{proof}[Proof of \autoref{thm:pbr_subseteq_2_closure}]
    We only need to show that $\overline{A}^\pbr \subseteq \overline{A}$.

    Let $t \in \overline{A}^\pbr$, so that there are $t_1, \ldots, t_l \in A$ such that $t = t_1 t_2 \ldots t_{2l-1}$, where for $l+1 \le r \le 2l-1$, we set $t_r = t_{2l-r}$, and for $0 \le r \le 2l-2$, the inequality \eqref{eqn:incr_len} holds. Thus, we are in the above setup, and we let $\alpha, \alpha_r \in \Phi^+$ be as above.

    We show by descending induction on $r$ that for $l-1 \ge r \ge 0$, $s_{t_r \ldots t_1(\alpha)} \in \overline{A}$. For $r = l-1$, by \autoref{lem:suff_prod} (i), $s_{t_{l-1} \ldots t_1(\alpha)} = t_l \in A$. Now suppose that $r < l-1$. We have $s_{\alpha_{r+1}} = t_{r+1} \in A$, and by the induction hypothesis, $s_{t_{r+1} \ldots t_1(\alpha)} \in \overline{A}$. By \autoref{lem:suff_prod} (ii) and (iii), and the fact that $\overline{A}$ is $2$-closed, we have $s_{t_r \ldots t_1(\alpha)} \in \overline{A}$, completing the induction.

    Taking $r = 0$, we get $t = s_\alpha \in \overline{A}$ as required.
\end{proof}

\subsection{Proof of \autoref{thm:simply_laced_pbr_eq_2_closure}}

In this section we give an explicit description of $2$-closures in the simply-laced case (\autoref{thm:2_closure_simply_laced}), then use it to prove \autoref{thm:simply_laced_pbr_eq_2_closure}. Assume for this section that $\Phi$ is a finite simply-laced root system. In particular, $\Phi$ is crystallographic, and has all roots of the same length. We normalize the symmetric bilinear form $(\cdot, \cdot)$ so that $(\alpha, \alpha) = 2$ for any root $\alpha \in \Phi$.

In the arguments that follow, we will often use the following lemma implicitly:
\begin{lemma}\label{lem:sum_is_root_iff}
    Assume that $\Phi$ is finite, simply-laced. Let $\alpha, \beta \in \Phi$. Then
    \begin{enumerate}[label=(\roman*)]
        \item $\alpha = \beta$ if and only if $(\alpha, \beta) = 2$;
        \item $\alpha = -\beta$ if and only if $(\alpha, \beta) = -2$;
        \item $\alpha + \beta \in \Phi$ if and only if $(\alpha, \beta) = -1$;
        \item $\alpha - \beta \in \Phi$ if and only if $(\alpha, \beta) = 1$.
    \end{enumerate}
\end{lemma}
\begin{proof}
    (i) and (ii) follow from the Cauchy--Schwarz inequality. (iv) follows by applying (iii) with $-\beta$ instead of $\beta$. Now we show (iii).

    If $\alpha + \beta \in \Phi$, then $2 = (\alpha+\beta, \alpha+\beta) = (\alpha, \alpha) + (\beta, \beta) + 2(\alpha, \beta) = 4 + 2(\alpha, \beta)$, so $(\alpha, \beta) = -1$. If $(\alpha, \beta) = -1$, then $s_\beta(\alpha) = \alpha - \langle \alpha, \beta^\vee \rangle \beta = \alpha - (\alpha, \beta) \beta = \alpha + \beta \in \Phi$.
\end{proof}

\begin{lemma}\label{lem:sum_is_root_implies_weak_le}
    Assume that $\Phi$ is finite, simply-laced. Let $\alpha, \beta \in \Phi^+$ be such that $\alpha + \beta \in \Phi^+$. Then $s_\alpha(\Phi_{s_\beta}) \subseteq \Phi^+$.
\end{lemma}
\begin{proof}
    By \cite[Lemma 4.1 (vi)]{MR3943754}, it suffices to show that $\Phi_{s_\alpha} \cap \Phi_{s_\beta} = \emptyset$. We can easily check that $\alpha, \beta, \alpha + \beta \notin \Phi_{s_\alpha} \cap \Phi_{s_\beta}$. Let $\gamma \in \Phi_{s_\alpha} \cap \Phi_{s_\beta}$. Then $(\gamma, \alpha) = 1$ and $(\gamma, \beta) = 1$, so $(\gamma, \alpha+\beta) = 2$, hence $\gamma = \alpha + \beta$, contradiction.
\end{proof}

We remark that in the simply-laced case, a subset $A \subseteq \Phi^+$ is $2$-closed if and only if for all $\alpha, \beta \in A$, $\alpha + \beta \in \Phi^+$ implies $\alpha + \beta \in A$. This can be seen from the fact that any rank $2$ subsystem of a simply-laced root system is either $A_1 \times A_1$ or $A_2$.

For $\alpha \in \Phi^+ \cup \{0\}$, we let $\Ht(\alpha)$ denote its \emph{height}. If $\alpha = \sum_{\beta \in \Delta} c_\beta \beta$, then by definition $\Ht(\alpha) = \sum_{\beta \in \Delta} c_\beta \in \ZZ_{\ge 0}$.

For $A \subseteq \Phi^+$, we define the relation $\to_A$ on $\Phi^+ \cup \{0\}$ by
\begin{equation*}
    \alpha \to_A \beta \qquad \text{if and only if} \qquad \alpha - \beta \in A.
\end{equation*}
We let $>_A$ denote the transitive closure of $\to_A$. Note that if $\alpha >_A \beta$, then $\Ht(\alpha) > \Ht(\beta)$. It follows that

\begin{lemma}
    For any $A \subseteq \Phi^+$, $>_A$ is a strict partial order on $\Phi^+ \cup \{0\}$.
\end{lemma}

In this setting, $2$-closures have the following explicit description:
\begin{theorem}\label{thm:2_closure_simply_laced}
    Assume that $\Phi$ is finite, simply-laced. Then for $A \subseteq \Phi^+$,
    \begin{equation*}
        \overline{A} = \{\beta \in \Phi^+ \mid \beta >_A 0\}.
    \end{equation*}
\end{theorem}
The set on the right-hand side generalizes the transitive closure of \cite{MR1267688}.

\begin{proof}[Proof of \autoref{thm:2_closure_simply_laced}]
    The inclusion $\{\beta \in \Phi^+ \mid \beta >_A 0\} \subseteq \overline{A}$ follows easily from induction on $\Ht(\beta)$ and the fact that $\overline{A}$ is $2$-closed.

    For the reverse inclusion, it suffices to show that if $\beta, \beta' >_A 0$ and $\beta + \beta' \in \Phi^+$, then $\beta + \beta' >_A 0$. We proceed by strong induction on $\Ht(\beta)$.

    If $\beta \in A$, then $\beta + \beta' \to_A \beta' >_A 0$ shows that $\beta + \beta' >_A 0$.

    Now suppose that $\beta \notin A$, so there is some $\beta_2 \in \Phi^+$ with $\beta \to_A \beta_2 >_A 0$. Note that $\Ht(\beta) > \Ht(\beta_2)$. Let $\alpha = \beta - \beta_2 \in A$. We claim that $\beta' \neq \alpha$. Indeed, by \autoref{lem:sum_is_root_iff}, we have $(\beta, \beta') = -1$ and $(\beta, \alpha) = 1$. Thus, $(\beta', \alpha) \in \{-1, 0, 1\}$.

    If $(\beta', \alpha) = 1$, then $(\beta + \beta', \alpha) = 2$, so $\beta + \beta' = \alpha \to_A 0$.

    If $(\beta', \alpha) = 0$, then $(\beta + \beta', \alpha) = 1$, so $\beta + \beta' - \alpha \in \Phi$, but we also have $\beta + \beta' - \alpha = \beta_2 + \beta'$, hence $\beta + \beta' - \alpha \in \Phi^+$. By the induction hypothesis, we have $\beta_2 + \beta' >_A 0$. Then $\beta + \beta' \to_A \beta_2 + \beta' >_A 0$.

    If $(\beta', \alpha) = -1$, then $\alpha + \beta' \in \Phi^+$. Since $\alpha \in A$, we have $\alpha + \beta' >_A 0$. By the induction hypothesis, we have $\beta + \beta' = \beta_2 + (\alpha + \beta') >_A 0$.
\end{proof}

\autoref{thm:simply_laced_pbr_eq_2_closure} follows from \autoref{thm:pbr_subseteq_2_closure_root} and
\begin{lemma}
    Assume that $\Phi$ is finite, simply-laced. Then for $A \subseteq \Phi^+$, we have $\overline{A} \subseteq \overline{A}^\pbr$.
\end{lemma}
\begin{proof}
    Let $\beta \in \overline{A}$. By \autoref{thm:2_closure_simply_laced}, there is a chain $\beta = \beta_1 \to_A \beta_2 \to_A \ldots \to_A \beta_l \to_A 0$. For $1 \le r \le l-1$, set $\alpha_r = \beta_r - \beta_{r+1} \in A$, and set $\alpha_l = \beta_l \in A$. For $1 \le r \le l$, set $t_r = s_{\alpha_r}$.

    We show by induction on $l$ that in this setup, we have
    \begin{enumerate}[label=(\roman*)]
        \item $\beta = t_1 \ldots t_{l-1}(\alpha_l)$; and
        \item for $0 \le r \le l-1$, $t_1 \ldots t_r(\alpha_{r+1}) \in \Phi_{s_\beta}$.
    \end{enumerate}
    This suffices by \autoref{lem:incr_len_iff}.

    For $l = 1$, we have $\alpha_1 = \beta_1 = \beta \in \Phi_{s_\beta}$, hence (i) and (ii) hold.

    Now suppose that $l > 1$. By induction, we may assume that
    \begin{enumerate}[label=(\roman*')]
        \item $\beta_2 = t_2 \ldots t_{l-1}(\alpha_l)$; and
        \item for $1 \le r \le l-1$, $t_2 \ldots t_r(\alpha_{r+1}) \in \Phi_{s_{\beta_2}}$.
    \end{enumerate}
    Then by (i'), $t_1 t_2 \ldots t_{l-1}(\alpha_l) = s_{\alpha_1}(\beta_2) = \beta_1 = \beta$, hence (i).

    We compute $s_\beta(\alpha_1) = -\beta_2$, so $\alpha_1 \in \Phi_{s_\beta}$. Let $1 \le r \le l-1$. We need to show that $t_1 \ldots t_r(\alpha_{r+1}) \in \Phi_{s_\beta}$.

    By \autoref{lem:sum_is_root_implies_weak_le}, $s_{\alpha_1}(\Phi_{s_{\beta_2}}) \subseteq \Phi^+$, so $t_1 \ldots t_r(\alpha_{r+1}) \in \Phi^+$ by (ii'). Furthermore,
    \begin{equation*}
        s_\beta t_1 t_2 \ldots t_r(\alpha_{r+1}) = t_1 s_{\beta_2} t_2 \ldots t_r(\alpha_{r+1}) = -s_{\alpha_1} (-s_{\beta_2} t_2 \ldots t_r(\alpha_{r+1})).
    \end{equation*}
    We have $-s_{\beta_2} t_2 \ldots t_r(\alpha_{r+1}) \in \Phi_{s_{\beta_2}}$ by (ii'), so $s_{\alpha_1} (-s_{\beta_2} t_2 \ldots t_r(\alpha_{r+1})) \in \Phi^+$. We conclude that $t_1 \ldots t_r(\alpha_{r+1}) \in \Phi_{s_\beta}$, hence (ii) holds as required.
\end{proof}

\section{Further directions}

The complete set of relations between the various closures in general is still not entirely clear. We initially conjectured the following:
\begin{conjecture}\label{conj:br_of_2_closed}
    For any $A \subseteq T$, we have $\overline{A}^\br \subseteq \overline{A}$.
\end{conjecture}
This is a strengthening of \autoref{thm:pbr_subseteq_2_closure} and \cite[Theorem 4.21]{dermenjian2025bruhatpreclosure}.

In \cite{dermenjian2025bruhatpreclosure}, Dermenjian made the following conjecture, which he has already proved in type $A$:
\begin{conjecture}[\protect{\cite[Conjecture 3.4]{dermenjian2025bruhatpreclosure}}]
    Bruhat preclosure is a closure operator for Coxeter groups of type $A$, $D$, or $E$.
\end{conjecture}

We can prove this assuming \autoref{conj:br_of_2_closed}, by showing that $2$-closure, Bruhat preclosure, and palindromic Bruhat preclosure are the same operation. Let $A \subseteq T$. Then
\begin{equation*}
    \overline{A} = \overline{A}^\pbr \subseteq \overline{A}^\br \subseteq \overline{A},
\end{equation*}
where the first equality is \autoref{thm:simply_laced_pbr_eq_2_closure}, and the second inclusion is \autoref{conj:br_of_2_closed}.

At present, the proof of Dyer's conjecture is done case-by-case. For a uniform proof of Dyer's conjecture for finite Coxeter groups, we propose the following for an arbitrary finite Coxeter group $W$:
\begin{conjecture}
    For $u, w \in W$, $\overline{\Phi_u \cup \Phi_w}^\pbr$ is $2$-closed.
\end{conjecture}

\section*{Acknowledgments}

The author learned about Dyer's conjecture from Christophe Hohlweg at the CIME school on Coxeter group theory and representation theory, held in Cetraro, Italy. The author is grateful for the school's hospitality. The author thanks Thomas Lam for enabling the author's Claude usage via Claude Team plan for scientists, and Charlotte Chan for continued support.

\section*{AI usage disclosure}

We used Gemini 3.1 Pro and Claude Opus 5 during this research, including literature search and proofreading. The proof of \autoref{thm:dyer_simply_laced} via results in \autoref{sec:pbr} was developed in close collaboration with Claude, who was able to supply proofs of many technical results under the author's direction. The main text is written entirely by the author, who takes full responsibility for the correctness of its results. We have made an effort to simplify many of the statements and arguments.

\appendix
\refstepcounter{section}
\phantomsection
\renewcommand{\thesubsection}{\Alph{subsection}}

\section*{Appendix: The join algorithm in classical types}\label{sec:joins_classical}

In the following three subsections, we specialize our algorithm in \autoref{subsec:algorithm} to types $A$, $B$, and $D$ in order. These subsections will have largely the same structure. For types $A$ and $B$, we recover the algorithms of Markowsky and Biagioli--Perrone respectively, as such we will be brief in the corresponding subsections. For type $D$, our algorithm is new.

For a finite Coxeter group $W$, we let $T_L(w) = \{t \in T \mid \ell(t w) < \ell(w) \}$ for $w \in W$. $\Phi_w$ and $T_L(w)$ are related by $T_L(w) = \{s_\alpha \mid \alpha \in \Phi_w\}$, where $s_\alpha \in T$ is the reflection with respect to $\alpha \in \Phi$.

\subsection*{Type A: Markowsky's algorithm}\label{subsec:algorithm_type_a}

Fix $n \ge 2$. Let $S_n$ denote the symmetric group on $[n] = \{1,2,\ldots,n\}$, which is a Coxeter group with generators $S = \{s_i \mid 1 \le i \le n-1\}$, where $s_i = (i \,\, i+1)$. The subset of reflections in $S_n$ is
\begin{equation*}
    T_n = \{(i \,\, j) \in S_n \mid 1 \le i < j \le n\}.
\end{equation*}
Set $I = \{s_i \mid 1 \le i \le n-2\}$, so that $(S_n)_I = S_{n-1}$.

\begin{lemma*}
    For $w \in S_n$, $\proj^S_I(w) \in S_{n-1}$ is the permutation obtained by removing $n$ from the one-line notation of $w$.
\end{lemma*}
\begin{proof}
    Let $w' \in S_{n-1}$ be the permutation obtained by removing $n$ from the one-line notation of $w$. By \autoref{lem:Phi_projection}, it suffices to show that $T_L(w') = T_L(w) \cap T_{n-1}$, which is clear.
\end{proof}

We now describe the fibers of $\proj^S_I$. A simple calculation and \autoref{rem:chain_iff} shows that for $w \in S_{n-1}$, the fiber $(\proj^S_I)^{-1}(\{w\}) = w \leftindex^I (S_n)$ is a chain in $S_n$. In fact, it consists of the $n$ elements
\begin{equation*}
    w_1 w_2 \cdots w_{n-1} n \le w_1 w_2 \cdots w_{n-2} n w_{n-1} \le \ldots \le n w_1 w_2 \cdots w_{n-1}
\end{equation*}
in one-line notation, where $w_i = w(i)$. In particular, it is linearly ordered by decreasing position of $n$. The following is a restatement of \autoref{lem:le_join_iff} in our setup:

\begin{lemma*}
    Let $u, w \in S_n$ and set $j' = \proj^S_I(u) \vee \proj^S_I(w)$. For $j \in (\proj^S_I)^{-1}(\{j'\})$, we have $u \vee w \le j$ if and only if $(T_L(u) \cup T_L(w)) \setminus T_{n-1} \subseteq T_L(j) \setminus T_{n-1}$.
\end{lemma*}

Thus to obtain the join $j = u \vee w$, we insert $n$ into $j' = \proj^S_I(u) \vee \proj^S_I(w)$ at the rightmost possible position such that the condition of the above lemma holds. This condition says that if $1 \le k \le n-1$ is such that $(k \,\, n) \in T_L(u) \cup T_L(w)$, then $(k \,\, n) \in T_L(j)$. In other words, if $k$ occurs to the right of $n$ in the one-line notation of $u$ or $w$, then it also occurs to the right of $n$ in $j$. This gives precisely the condition in \cite[Theorem 2]{MR1267688}, and a proof of its correctness.

\subsection*{Type B: the Biagioli--Perrone algorithm}\label{subsec:algorithm_type_b}

Fix $n \ge 2$. Let $S^B_n$ denote the signed permutation group on $[n]$, which is a Coxeter group with generators $S = \{s_i \mid 0 \le i \le n-1 \}$, where $s_0 = (\overline{1} \,\, 1)$, and for $1 \le i \le n-1$, $s_i = (\overline{i+1} \,\, \overline{i}) (i \,\, i+1)$. Here, we write $\overline{i}$ for $-i$. We refer the reader to \cite[Section 8.1]{bjorner2005combinatorics} for more details. The subset of reflections in $S^B_n$ is
\begin{equation*}
    T^B_n = \{(\overline{j} \,\, \overline{i})(i \,\, j) \in S^B_n \mid 1 \le \abs{i} < j \le n \} \cup \{(\overline{i} \,\, i) \mid 1 \le i \le n \}.
\end{equation*}
Set $I = \{s_i \mid 0 \le i \le n-2\}$, so that $(S^B_n)_I = S^B_{n-1}$.

\begin{lemma*}
    For $w \in S^B_n$, $\proj^S_I(w) \in S^B_{n-1}$ is the signed permutation obtained by removing $n$ and $\overline{n}$ from the complete notation of $w$.
\end{lemma*}
\begin{proof}
    Let $w' \in S^B_{n-1}$ be the signed permutation obtained by removing $n$ and $\overline{n}$ from the complete notation of $w$. By \autoref{lem:Phi_projection}, it suffices to show that $T_L(w') = T_L(w) \cap T^B_{n-1}$, which is clear.
\end{proof}

We now describe the fibers of $\proj^S_I$. A simple calculation and \autoref{rem:chain_iff} shows that for $w \in S^B_{n-1}$, the fiber $(\proj^S_I)^{-1}(\{w\}) = w \leftindex^I (S^B_n)$ is a chain in $S^B_n$. In fact, it consists of the $2n$ elements
\begin{align*}
    &w_1 w_2 \cdots w_{n-1} n \le w_1 w_2 \cdots w_{n-2} n w_{n-1} \le \ldots \le n w_1 w_2 \cdots w_{n-1} \\
    &\qquad\le \overline{n} w_1 w_2 \cdots w_{n-1} \le w_1 \overline{n} w_2 \cdots w_{n-1} \le \ldots \le w_1 w_2 \cdots w_{n-1} \overline{n}
\end{align*}
in window notation. In particular, it is linearly ordered by decreasing position of $n$ in complete notation. The following is a restatement of \autoref{lem:le_join_iff} in our setup:

\begin{lemma*}
    Let $u, w \in S^B_n$ and set $j' = \proj^S_I(u) \vee \proj^S_I(w)$. For $j \in (\proj^S_I)^{-1}(\{j'\})$, we have $u \vee w \le j$ if and only if $(T_L(u) \cup T_L(w)) \setminus T^B_{n-1} \subseteq T_L(j) \setminus T^B_{n-1}$.
\end{lemma*}

The full description of the type $B$ algorithm can be found in \cite{biagioli2026computing}. We omit its rederivation here, which is very similar to the type $A$ case. We leave it as an exercise for the interested reader.

\subsection*{Type D}\label{subsec:algorithm_type_d}

Finally, we work out the algorithm for type $D$.

For $n \ge 2$, let $S^D_n$ denote the even signed permutation group on $[n]$, which is a Coxeter group with generators $S = \{s_i \mid 0 \le i \le n-1 \}$, where $s_0 = (\overline{2} \,\, 1) (\overline{1} \,\, 2)$, and for $1 \le i \le n-1$, $s_i = (\overline{i+1} \,\, \overline{i}) (i \,\, i+1)$. We refer the reader to \cite[Section 8.2]{bjorner2005combinatorics} for more details. The subset of reflections in $S^D_n$ is
\begin{equation*}
    T^D_n = \{(\overline{j} \,\, \overline{i})(i \,\, j) \in S^D_n \mid 1 \le \abs{i} < j \le n \}.
\end{equation*}
Set $I = \{s_i \mid 0 \le i \le n-2\}$, so that for $n \ge 3$, $(S^D_n)_I = S^D_{n-1}$.

Now suppose $n \ge 3$.
\begin{lemma*}
    For $w \in S^D_n$, $\proj^S_I(w) \in S^D_{n-1}$ is the even signed permutation obtained by removing $n$ and $\overline{n}$ from the complete notation of $w$, then flipping the signs of the two middle elements if $w^{-1}(n) < 0$.
\end{lemma*}
\begin{proof}
    Let $w' \in S^B_{n-1}$ be the signed permutation obtained by the above procedure. By construction, the sign flipping ensures that $w' \in S^D_{n-1}$. By \autoref{lem:Phi_projection}, it suffices to show that $T_L(w') = T_L(w) \cap T^D_{n-1}$, which is clear.
\end{proof}

The fibers of $\proj^S_I$ are no longer chains, but are still easy to describe. For $w \in S^D_{n-1}$, the fiber $(\proj^S_I)^{-1}(\{w\}) = w \leftindex^I (S^D_n)$ consists of the $2n$ elements
\begin{align*}
    &w_1 w_2 \cdots w_{n-1} n \le w_1 w_2 \cdots w_{n-2} n w_{n-1} \le \ldots \le n w_1 w_2 \cdots w_{n-1}, \overline{n} \overline{w_1} w_2 \cdots w_{n-1} \\
    &\qquad \le \overline{w_1} \overline{n} w_2 \cdots w_{n-1} \le \ldots \le \overline{w_1} w_2 \cdots w_{n-1} \overline{n}
\end{align*}
in window notation. Note that the two elements $n w_1 w_2 \cdots w_{n-1}$ and $\overline{n} \overline{w_1} w_2 \cdots w_{n-1}$ are incomparable. The weak order on the fiber can be extended to a linear ordering by decreasing position of $n$ in complete notation. The following is a restatement of \autoref{lem:le_join_iff} in our setup:

\begin{lemma*}
    Let $u, w \in S^D_n$ and set $j' = \proj^S_I(u) \vee \proj^S_I(w)$. For $j \in (\proj^S_I)^{-1}(\{j'\})$, we have $u \vee w \le j$ if and only if $(T_L(u) \cup T_L(w)) \setminus T^D_{n-1} \subseteq T_L(j) \setminus T^D_{n-1}$.
\end{lemma*}

The above condition says that if $k \in \pm [n-1]$ occurs to the right of $n$ in the complete notation in $u$ or $w$, then it also occurs to the right of $n$ in $j$. Thus to obtain the join $j = u \vee w$, we insert $n$ into $j' = \proj^S_I(u) \vee \proj^S_I(w)$ at the rightmost possible position such that this condition holds. If we cannot insert $n$ in the right half of $j'$, then we also flip the signs of the middle two elements of $j'$ before trying to insert $n$ from right to left in the left half of this modified $j'$. To avoid degeneracies, our base case here will be $n=2$, for which we can compute joins directly.

\begin{example}
    Let $u = \typeb{-2,5,4,-3,1}$ and $w = \typeb{3,1,4,5,2}$. Set $u_5 = u, w_5 = w$, and compute recursively $u_i = \proj^{\{s_0, \ldots, s_i\}}_{\{s_0, \ldots, s_{i-1}\}} u_{i+1}$ and $w_i = \proj^{\{s_0, \ldots, s_i\}}_{\{s_0, \ldots, s_{i-1}\}} w_{i+1}$ for $4 \ge i \ge 2$. We then directly calculate $u_2 \vee w_2$. Having calculated $u_i \vee w_i$, we then use the above procedure to insert $i+1$ into $u_i \vee w_i$. The following table shows the intermediate calculations:
    \begin{center}
    \begin{tabular}{|C|C|C|C|}
        \hline
        i & u_i & w_i & u_i \vee w_i \\
        \hline
        5 & \typeb{-2,5,4,-3,1} & \typeb{3,1,4,5,2} & \typeb{-5,4,2,-3,1} \\
        4 & \typeb{-2,4,-3,1} & \typeb{3,1,4,2} & \typeb{-4,2,-3,1} \\
        3 & \typeb{-2,-3,1} & \typeb{3,1,2} & \typeb{-2,-3,1} \\
        2 & \typeb{2,1} & \typeb{1,2} & \typeb{2,1} \\
        \hline
    \end{tabular}
    \end{center}
    For example, to get $u_2$ from $u_3$, we first delete $3$ from $u_3$. Since $3$ is in the left half of $u_3$, we also need to flip the sign of the middle elements $\pm 2$.

    To get $u_3 \vee w_3$ from $j_2 = u_2 \vee w_2$, we need to insert $3$ in a suitable position. From $u_3$ and $w_3$, we see that $3$ has to be inserted into $j_2$ to the left of $\overline{2}, 1, 2$, hence we must insert $3$ into the left half of $j_2$, so the sign of the middle element $2$ has to be flipped first, to obtain $\typeb{-2,1}$. We insert $3$ into the left side of this, trying from right to left, to get $\typeb{-2,-3,1}$.
\end{example}

\printbibliography

\end{document}